\documentclass[11pt]{article}
\usepackage{etex}
\reserveinserts{28}
\usepackage[all]{xy}

\usepackage{amsfonts}
\usepackage{amsthm}
\usepackage{enumerate}
\usepackage{graphicx}
\usepackage{mathrsfs}
\usepackage{bm}
\usepackage{cite}
\usepackage{amssymb,amsmath} % font used for R in Real numbers
\usepackage{makecell}
\usepackage{tikz}
\usepackage{pgf}
\usepackage{tikz}
\usetikzlibrary{patterns}
\usepackage{pgffor}
\usepackage{pgfcalendar}
\usepackage{pgfpages}
\usepackage{shuffle,yfonts}
\usepackage{mathtools}
\DeclareFontFamily{U}{shuffle}{}
\DeclareFontShape{U}{shuffle}{m}{n}{ <-8>shuffle7 <8->shuffle10}{}

\newcommand{\CMZV}{\mathsf{CMZV}}
\newcommand{\sha}{\shuffle}

\newcommand{\ola}{\overleftarrow}
\newcommand{\ora}{\overrightarrow}

\newcommand\tx{{\texttt{x}}}
\newcommand\ty{{\texttt{y}}}
\newcommand\tz{{\texttt{z}}}
\newcommand\ttx{{\emph{\texttt{x}}}}
\newcommand{\yi}{{1}}

\newcommand\om{{\omega}}

\newcommand\eps{{\varepsilon}}

\newcommand{\bfk}{{\boldsymbol{\sl{k}}}}

\newcommand{\bfm}{{\boldsymbol{\sl{m}}}}

\newcommand{\bfz}{{\boldsymbol{\sl{z}}}}

\newcommand\bfsi{{\boldsymbol \sigma}}

 \allowdisplaybreaks
\usetikzlibrary{arrows,shapes,chains}
 \allowdisplaybreaks

\catcode`!=11
\let\!int\int \def\int{\displaystyle\!int}
\let\!lim\lim \def\lim{\displaystyle\!lim}
\let\!sum\sum \def\sum{\displaystyle\!sum}
\let\!sup\sup \def\sup{\displaystyle\!sup}
\let\!inf\inf \def\inf{\displaystyle\!inf}
\let\!cap\cap \def\cap{\displaystyle\!cap}
\let\!max\max \def\max{\displaystyle\!max}
\let\!min\min \def\min{\displaystyle\!min}
\let\!frac\frac \def\frac{\displaystyle\!frac}
\catcode`!=12

\let\oldsection\section
\renewcommand\section{\setcounter{equation}{0}\oldsection}

\allowdisplaybreaks

\DeclareMathOperator*{\dep}{dep}
\DeclareMathOperator{\Li}{Li}

\DeclareMathOperator{\ti}{ti}

\def\N{\mathbb{N}}
\def\Z{\mathbb{Z}}

\def\su{\sum\limits_{n=1}^\infty}

\def\ze{\zeta}

\theoremstyle{plain}
\newtheorem{thm}{Theorem}[section]
\newtheorem{lem}[thm]{Lemma}
\newtheorem{cor}[thm]{Corollary}

\newtheorem{pro}[thm]{Proposition}
\theoremstyle{definition}
\newtheorem{defn}{Definition}[section]
\newtheorem{re}[thm]{Remark}
\newtheorem{exa}[thm]{Example}

\begin{document}
%%%%%%%%%%%%%%%%%%%% title %%%%%%%%%%%%%%%%%%%%%%%%%%%%%%%%%%%%%%%%%%%%%%%%
\title{\bf Explicit Evaluation of Euler-Ap\'ery Type Multiple Zeta Star Values and Multiple $t$-Star Values}
\author{
{Ce Xu${}^{a,}$\thanks{Email: cexu2020@ahnu.edu.cn} and Jianqiang Zhao${}^{b,}$\thanks{Email: zhaoj@ihes.fr}}\\[1mm]
\small a. School of Mathematics and Statistics, Anhui Normal University, Wuhu 241002, PRC\\
\small b. Department of Mathematics, The Bishop's School, La Jolla, CA 92037, USA}

\date{}
\maketitle

\noindent{\bf Abstract.} In this paper we establish several recurrence relations about Euler-Ap\'ery type multiple zeta star values and a parametric variant of it by using the method of iterated integrals. Then using the formulas obtained, we find the explicit evaluations for some specific Euler-Ap\'ery type multiple zeta star values and one of its parametric variant, and Euler-Ap\'ery type multiple $t$-star values.



\section{Definition}
We begin with some basic notations. Let $\Z$ and $\N$ be the set of integers and positive integers, respectively.  A finite sequence $\bfk:= (k_1,\ldots, k_r)\in\N^r$ is called a \emph{composition}. We put
\[|\bfk|:=k_1+\cdots+k_r,\quad \dep(\bfk):=r,\]
and call them the weight and the depth of $\bfk$, respectively. If $k_1>1$, $\bfk$ is called \emph{admissible}.

For a composition $\bfk=(k_1,\ldots,k_r)$ and positive integer $n$, the \emph{multiple harmonic sums} (MHSs) and \emph{multiple harmonic star sums} (MHSSs) are defined by
\begin{align}
\zeta_n(\bfk):=\sum\limits_{n\geq n_1>\cdots>n_r>0 } \frac{1}{n_1^{k_1}\cdots n_r^{k_r}}\quad
\text{and}\quad
\zeta^\star_n(\bfk):=\sum\limits_{n\geq n_1\geq\cdots\geq n_r>0} \frac{1}{n_1^{k_1}\cdots n_r^{k_r}}\label{MHSs+MHSSs},
\end{align}
respectively. If $n<k$ then ${\zeta_n}(\bfk):=0$ and ${\zeta _n}(\emptyset )={\zeta^\star _n}(\emptyset ):=1$. As special cases,
$$
H_n:= \ze_n(1)=\ze^\star_n(1)  \quad  \text{and}  \quad H_n^{(k)}:=\ze_n(k)=\ze^\star_n(k)
$$
are the \emph{classical} and \emph{generalized harmonic numbers}, respectively. When taking the limit $n\rightarrow \infty$ in \eqref{MHSs+MHSSs} we get the so-called  \emph{multiple zeta values} (MZVs) and the \emph{multiple zeta star values} (MZSVs), respectively:
\begin{align*}
{\zeta}( \bfk):=\lim_{n\rightarrow \infty}{\zeta _n}(\bfk), \quad
\text{and}\quad
{\zeta^\star}( \bfk):=\lim_{n\rightarrow \infty}{\zeta^\star_n}( \bfk),
\end{align*}
defined for admissible compositions  $\bfk$ to ensure convergence of the series. The systematic study of MZVs began in the early 1990s with the works of Hoffman \cite{H1992} and Zagier \cite{DZ1994}. Due to their surprising and sometimes mysterious appearance in the study of many branches of mathematics and theoretical physics, these special values have attracted a lot of attention and interest in the past three decades (for example, see the book by the
second author  \cite{Zhao2016}).

In \cite{H2019}, Hoffman introduced and studied odd variants of MZVs and MZSVs, which are
defined for an admissible composition $\bfk=(k_1,k_2,\ldots,k_r)$ by
\begin{align}\label{defn-mtvs}
t(\bfk):=\sum_{n_1>n_2>\cdots>n_r>0} \prod_{j=1}^r \frac{1}{(2n_{j}-1)^{k_{j}}}\quad
\text{and}\quad
t^\star(\bfk):=\sum_{n_1\geq n_2\geq \cdots\geq n_r>0} \prod_{j=1}^r \frac{1}{(2n_{j}-1)^{k_{j}}},
\end{align}
and are called a \emph{multiple $t$-value} and \emph{multiple $t$-star value}, respectively.

Then, similar to multiple harmonic sums and multiple harmonic star sums, for a composition $\bfk=(k_1,\ldots,k_r)$ and positive integer $n$, we define the \emph{multiple t-harmonic sums} and \emph{multiple t-harmonic star sums} respectively by
\begin{align*}
t_n(\bfk):=\sum_{n\geq n_1>n_2>\cdots>n_r>0} \prod_{j=1}^r \frac{1}{(2n_{j}-1)^{k_{j}}}\quad
\text{and}\quad
t^\star_n(\bfk):=\sum_{n\geq n_1\geq n_2\geq \cdots\geq n_r>0} \prod_{j=1}^r \frac{1}{(2n_{j}-1)^{k_{j}}}.
\end{align*}

In general, for any ${\bfk}=(k_1,\dotsc,k_r)\in\N^r$ and $N$th roots of unity $z_1,\dotsc,z_r$, we can define the \emph{colored MZVs} (CMZVs) of level $N$ as
\begin{equation}\label{equ:defnMPL}
\Li_{k_1,\dotsc,k_r}(z_1,\dotsc,z_r):=\sum_{n_1>\cdots>n_r>0}
\frac{z_1^{n_1}\dots z_r^{n_r}}{n_1^{k_1} \dots n_r^{k_r}}
\end{equation}
which converges if $(s_1,z_1)\ne (1,1)$ (see \cite[Ch. 15]{Zhao2016}), in which case we call $({\bf s};{\bf z})$ \emph{admissible}. In particular, if all $z_j\in \{\pm 1\}$ in \eqref{equ:classicalLi}, then the level two colored MZVs are called \emph{alternating MZVs} (or \emph{Euler sums}). In this case, namely, when $(z_1,\ldots,z_r)\in \{\pm 1\}^r$ and $(k_1,z_1)\neq (1,1)$, we set $\ze(k_1,\ldots,k_r;z_1,\ldots,z_r)=\Li_{k_1,\dotsc,k_r}(z_1,\dotsc,z_r)$. Further, we put a bar on top of
$k_j$ if $z_j=-1$. For example,
\begin{equation*}
\zeta(\bar 2,3,\bar 1,4)=\zeta(  2,3,1,4;-1,1,-1,1).
\end{equation*}
More generally, for any composition $(k_1,\dotsc,k_r)\in\N^r$, the \emph{classical multiple polylogarithm function} with $r$-variable is defined by
\begin{align}\label{equ:classicalLi}
\Li_{k_1,\dotsc,k_r}(x_1,\dotsc,x_r):=\sum_{n_1>n_2>\cdots>n_r>0} \frac{x_1^{n_1}\dotsm x_r^{n_r}}{n_1^{k_1}\dotsm n_r^{k_r}}
\end{align}
which converges if $|x_j\cdots x_r|<1$ for all $j=1,\dotsc,r$. They can be analytically continued to a multi-valued meromorphic function on $\mathbb{C}^r$ (see \cite{Zhao2007d}). In particular, if $x_1=x,x_2=\cdots=x_r=1$, then $\Li_{k_1,\ldots,k_r}(x,\yi_{r-1})$ is the classical multiple polylogarithm function with single-variable. As a convention, we denote by $\yi_d$ the sequence of 1's with $d$ repetitions. We also define the \emph{multiple t-polylogarithm function} for any composition $\bfk=(k_1,\ldots,k_r)$ by
\begin{align*}
\ti_{\bfk}(x)&:=\sum_{n_1>n_2>\cdots>n_r>0} \frac{x^{2n_1-1}}{(2n_1-1)^{k_1}(2n_2-1)^{k_2}\ldots (2n_r-1)^{k_r}}\quad \nonumber\\
&=\int_0^x \frac{dt}{1-t^2}\left(\frac{dt}{t}\right)^{k_r-1} \frac{tdt}{1-t^2}\left(\frac{dt}{t}\right)^{k_{r-1}-1}\cdots \frac{tdt}{1-t^2}\left(\frac{dt}{t}\right)^{k_1-1},
\end{align*}
where $|x|\le 1$ with $(k_1,x)\ne (1,1)$. Clearly, $t(\bfk;1)=t(\bfk)$ with $k_1\geq 2$.

Motivated by \cite{Au2020}, Xu and Wang \cite{WX2021} studied some Euler-Ap\'{e}ry type series of the form
\begin{equation}\label{EAS.def}
\sum_{n=1}^\infty
    \frac{H_n^{(k_1)}H_n^{(k_2)}\cdots H_n^{(k_r)}}{n^p}a_n^{\pm 1}
\end{equation}
where for $a_n=\binom{2n}{n}/4^n$. For these and some other similar series they established
the corresponding explicit formulas in terms of the alternating MZVs.
In particular, they discovered a few elegant explicit formulas for the series
\begin{equation*}
\sum_{n=1}^{\infty}\frac{a_n}{n^p}\,,\quad
\sum_{n=1}^{\infty}\frac{a_n H_n^{(m)}}{n^p}\,,\quad
\sum_{n=1}^{\infty}\frac{a_nH_nH_n^{(m)}}{n^p}\,,\quad
\sum_{n=1}^{\infty}\frac{a_nH_n^3}{n^p}\,,
\end{equation*}
and
\begin{equation*}
\sum_{n=1}^\infty\frac{a_n\ze_n^\star(1_m)}{n^p}\,,\quad
\sum_{n=1}^\infty\frac{a_nH_n\ze_n^\star(1_m)}{n^p}\,,
\end{equation*}
for $m,p\geq1$, by using the method of iterated integrals. They also
found some expressions of the Euler-Ap\'{e}ry type series
\begin{equation*}
\sum_{n=1}^\infty\frac{a_n^{-1}}{n^p}\,,\quad
\sum_{n=1}^\infty\frac{a_n^{-1}H_n}{n^p}\,,\quad
\sum_{n=1}^\infty\frac{a_n^{-1}H_{2n}}{n^p}\,,\quad
\sum_{n=1}^\infty\frac{a_n^{-1}O_n}{n^p}\,,
\end{equation*}
for $p\geq 2$, and
\begin{equation*}
\sum_{n=1}^{\infty}\frac{a_n}{n^p}\,,\quad
\sum_{n=1}^\infty\frac{a_nH_{2n}}{n^p}\,,\quad
\sum_{n=1}^\infty\frac{a_nO_n}{n^p}\,,\quad
\end{equation*}
for $p\geq 1$, by computing the contour integrals related to gamma functions, polygamma functions and
trigonometric functions. Here $O_n=\sum_{k=1}^n\frac{1}{2k-1}$ are the classical \emph{odd harmonic numbers}. Obviously, by applying the stuffle relations, also called quasi-shuffle relations (see
\cite{Hoffman2000}), we know
that for any composition $\bfk=(k_1,\ldots,k_r)$, the product $H_n^{(k_1)}\cdots H_n^{(k_r)}$ can be expressed in terms of a linear combination of multiple harmonic (star) sums (for the explicit formula, see \cite[Eq. (2.4)]{WX2020}). For example
\begin{equation*}
H_nH_n^{(2)}=\ze_n(1)\ze_n(2)=\ze_n(1,2)+\ze_n(2,1)+\ze_n(3).
\end{equation*}
Hence, we can study the Euler-Ap\'ery type MZVs
\begin{equation*}
\sum_{n=1}^\infty \frac{\zeta_n(k_2,\ldots.k_r)}{n^{k_1}}a_n^{\pm 1}
\end{equation*}
to obtain some explicit evaluations of \eqref{EAS.def}. Recently, Au \cite{Au2020} proved the result that for ${\bfk}=(k_1,\ldots,k_r)\in \N^r$, the Euler-Ap\'ery type MZVs above can be expressed in terms of alternating {\rm MZVs} (but he did not give a general explicit formula), namely,
\begin{align*}
\sum_{n=1}^\infty \frac{\zeta_n(k_2,\ldots.k_r)}{n^{k_1}}a_n^{\pm 1} \in \CMZV^2_{|{\bfk}|}.
\end{align*}
Therefore, the Euler-Ap\'{e}ry type series \eqref{EAS.def} can be evaluated by alternating MZVs. Some related results may be found in \cite{Camp18,ChenKW19,KWY2007,X2020} and references therein. Further, Au showed that
\begin{align*}
\sum_{n=1}^\infty \frac{\zeta_n(k_2,\ldots.k_r)}{n^{k_1}}a_n^{-2} \in \CMZV^4_{|{\bfk}|},\quad \sum_{n=1}^\infty \frac{\zeta_n(k_2,\ldots.k_r)}{n^{k_1}}a_n^2 \in \frac{1}{\pi}\CMZV^4_{|{\bfk}|+1},
\end{align*}
where $\CMZV^N_n$ is the $\mathbb{Q}$-span of CMZVs of weight $n$ and level $N$.

In this paper, we will study the following Euler-Ap\'ery type MZSVs and MtSVs
\begin{align}
&\sum\limits_{n=1}^\infty \frac{\zeta_n^\star(k_2,\ldots,k_r)}{n^{k_1}}a_n=\sum\limits_{n=1}^\infty \frac{\zeta_n^\star(k_2,\ldots,k_r)}{n^{k_1}4^{n}}\binom{2n}{n},\label{EAT-MZSVs-def}\\
&\sum\limits_{n=1}^\infty \frac{t_n^\star(k_2,\ldots,k_r)}{n^{k_1}}a_n=\sum\limits_{n=1}^\infty \frac{t_n^\star(k_2,\ldots,k_r)}{n^{k_1}4^{n}}\binom{2n}{n},\label{EAT-MtSVs-def}
\end{align}
and a parametric variant of \eqref{EAT-MZSVs-def}
\begin{align}\label{EAT-MZSVs-def2}
\sum\limits_{n=1}^\infty \frac{\zeta_n^\star(k_2,\ldots,k_r;x)}{n^{k_1}}a_n=
\sum\limits_{n=1}^\infty \frac{\zeta_n^\star(k_2,\ldots,k_r;x)}{n^{k_1}4^{n}}\binom{2n}{n},
\end{align}
where the \emph{parametric multiple harmonic star sum} (PMHSS) $\zeta^\star_n(k_1,\cdots,k_r;x)$ is defined by
\begin{align*}
\zeta_n^\star(k_1,\cdots,k_r;x):=\sum\limits_{n\geq n_1\geq \cdots \geq n_r\geq 1}\frac{x^{n_r}}{n^{k_1}_1\cdots n^{k_r}_r}\quad\text{and}\quad \zeta^\star_n(\emptyset;x):=x^n.
\end{align*}

The primary goal of this paper is to study the explicit relations of \eqref{EAT-MZSVs-def} and \eqref{EAT-MZSVs-def2}.
We will use the method of iterated integrals to obtain some recurrence relations of \eqref{EAT-MZSVs-def} and \eqref{EAT-MZSVs-def2},
which in turn will lead to some explicit evaluations of \eqref{EAT-MZSVs-def} and \eqref{EAT-MZSVs-def2}.

\section{Euler-Ap\'ery type MZSVs and its parametric variant}\label{EATMZVs}
The theory iterated integrals was developed first by K.T. Chen in the 1960's \cite{KTChen1971,KTChen1977}. It has played important roles in the study of algebraic topology and algebraic geometry in past half century. Its simplest form is
$$\int_{a}^b f_p(t)dt \cdots f_1(t)dt:=\int\limits_{a<t_p<\cdots<t_1<b}f_p(t_p) \cdots f_1(t_1)\, dt_1 \cdots dt_p.$$
In this section, we use the iterated integrals to establish two recurrence relations of \eqref{EAT-MZSVs-def} and \eqref{EAT-MZSVs-def2}. First, by \cite[Eqs. (3.1) and (3.2)]{XuZhao2020b}, we have the iterated integral expression
\begin{align}\label{equ:glInteratedInt}
\Li_{k_1,\ldots,k_r}(x_1,x_2/x_1\dotsc,x_r/x_{r-1})= \int_0^1 \left(\frac{x_r\, dt}{1-x_rt}\right)\left(\frac{dt}{t}\right)^{k_r-1}\cdots
\left(\frac{x_1\, dt}{1-x_1 t}\right)\left(\frac{dt}{t}\right)^{k_1-1}.
\end{align}
In particular, CMZVs can be expressed using iterated integrals
\begin{equation}\label{czeta-itIntExpression}
\Li_{\bfk}(\bfz)=\int_0^1 \tx_{\xi_r}\tx_0^{k_r-1}\cdots\tx_{\xi_1}\tx_0^{k_1-1},
\end{equation}
where $\xi_j:=\prod_{i=1}^j z_i^{-1}$,  and $\tx_\xi=dt/(\xi-t)$
for any $N$th roots of unity $\xi$, see \cite[Sec.~2.1]{Zhao2016} for a brief summary of this theory.

To save space, for any composition $\bfm=(m_1,\dotsc,m_p)\in\N^p$ and $i,j\in\N$, we put
\begin{align*}
&{\ora\bfm}_{\hskip-1pt i,j}:=
\left\{
  \begin{array}{ll}
    (m_i,\ldots,m_{j}), \quad \ & \hbox{if $i\le j\le p$;} \\
    \emptyset, & \hbox{if $i>j$,}
  \end{array}
\right.
 \quad &\ola\bfm_{\hskip-1pt i,j}:=
\left\{
  \begin{array}{ll}
     (m_{j},\ldots,m_i), \quad\ & \hbox{if $i\le j\le p$;} \\
     \emptyset, & \hbox{if $i>j$.}
  \end{array}
\right.
\end{align*}
Set $\ora\bfm_{\hskip-1pt i}=\ora\bfm_{\hskip-1pt 1,i}=(m_1,\ldots,m_i)$ and $\ola\bfm_{\hskip-1pt i}=\ola\bfm_{\hskip-1pt i,p}=(m_p,\ldots,m_i)$ for all $1\le i\le p$. Further set
\begin{align*}
\bfm_-=(m_1,\dotsc,m_p-1) \quad\text{if $m_p>1$, and}\quad \bfm_+=(m_1,\dotsc,m_p+1).
\end{align*}

The Hoffman dual of a composition $\bfm=(m_1,\ldots,m_p)$ is $\bfm^\vee=(m'_1,\ldots,m'_{p'})$ determined by
$|\bfm|:=m_1+\cdots+m_p=m'_1+\cdots+m'_{p'}$ and
\begin{equation*}
\{1,2,\ldots,|\bfm|-1\}
=\Big\{ \sideset{}{_{i=1}^{j}}\sum m_i\Big\}_{j=1}^{p-1}
 \coprod \Big\{ \sideset{}{_{i=1}^{j}}\sum  m_i'\Big\}_{j=1}^{p'-1}.
\end{equation*}
Equivalently, $\bfm^\vee$ can be obtained from $\bfm$ by swapping the commas ``,'' and the plus signs ``+'' in the expression
\begin{equation*}
 \bfm=(\underbrace{1+\cdots+1}_{\text{$m_1$ times}},\dotsc,\underbrace{1+\cdots+1}_{\text{$m_p$ times}}).
\end{equation*}
For example, we have
$({1,1,2,1})^\vee=(3,2)$ and $({1,2,1,1})^\vee=(2,3).$ More general, we have
\begin{align}\label{eq:HdualDef}
{\bfm}^\vee=(\underbrace{1,\ldots,1}_{m_1}+\underbrace{1,\ldots,1}_{m_2}+1,\ldots,1+\underbrace{1,\ldots,1}_{m_p}).
\end{align}
Put $\tx_1=dt/(1-t)$ and $\tx_0=dt/t$. Concerning this duality,
from \cite[Eq. (2.8)]{XuZhao2021a}, we have the iterated integral expression
\begin{equation}\label{IMP-2}
\zeta_n^\star(\bfm^\vee;x)
-\sum\limits_{j=1}^p (-1)^{p-j}\zeta^\star_n(\ora\bfm_j^\vee){\Li}_{\ola\bfm_{j+1}}(1-x)
=n(-1)^{p}\int_x^1\tx_1^{m_p-1}\tx_0 \cdots  \tx_1^{m_1}\tx_0t^{n-1}dt.
\end{equation}

To state our result, we need the following lemmas. To save space, we set  $\tx_\xi:=\frac{dt}{\xi-t}\quad (\xi\neq 0)$, $\ty=\tx_{-i}+\tx_{i}-\tx_{-1}-\tx_{1}$, $\tz=-\tx_0-\tx_{-i}-\tx_{i}$, $\om_0:=\frac{dt}{1-t^2}$ and $\om_1:=\frac{tdt}{1-t^2}$.

\begin{lem}\label{thm-ItI1} \emph{(cf.\ \cite[Thm. 2.1]{XuZhao2021a})}
For any $n,p\in\N$, $\bfm=(m_1,\ldots,m_p)\in\N^p$ and $x\in[0,1]$,
\begin{align}\label{FII1}
n\int_0^x t^{n-1}dt \ttx_1\ttx_0^{m_1-1}\cdots \ttx_1\ttx_0^{m_p-1}%\nonumber\\
&=(-1)^p\ze^\star_n(\ora\bfm;x)- \sum_{j=1}^{p}(-1)^{j} \ze^\star_n(\ora\bfm_{j-1})\Li_{\ola\bfm_{j}}(x).
\end{align}
\end{lem}

\begin{lem}\label{thm-ItI11} \emph{(cf.\ \cite[Thm. 3.6]{XuZhao2021a})}
For any $n,p\in\N$, $\bfm=(m_1,\ldots,m_p)\in\N^p$ and $x\in[0,1]$,
\begin{align}\label{FIIt1}
2n \int_0^x t^{2n-1}dt \om_0 \ttx_0^{m_1-1} \om_1 \ttx_0^{m_2-1} \cdots \om_1\ttx_0^{m_p-1}
=(-1)^p t^\star_n(\bfm;x)-\sum_{j=1}^p (-1)^{j} t^\star_n(\ora\bfm_{\hskip-2pt j-1})\ti_{\ola\bfm_{\hskip-2pt j}}(x),
\end{align}
where
\begin{align}\label{equ:t-mhss}
t^\star_n(\bfk;x):=\sum_{n\geq n_1\geq n_2\geq \cdots \geq n_r\geq 1} \frac{x^{2n_r-1}}{(2n_1-1)^{k_1}(2n_2-1)^{k_2}\ldots(2n_r-1)^{k_r}}.
\end{align}

\end{lem}

\begin{thm}\label{EATMZV1} For any $k,p\in\N$, $\bfm=(m_1,\ldots,m_p)\in\N^p$ with $m_p\geq 2$, we have
\begin{align*}
&\sum_{j=1}^{p} (-1)^{j} \zeta(\ola\bfm_{\hskip-1pt j})\sum_{n=1}^\infty \frac{\zeta_n^\star(\ora\bfm_{j-1})}{n^{k}4^n} \binom{2n}{n}
-(-1)^p \sum_{n=1}^\infty \frac{\zeta_n^\star(\bfm)}{n^{k}4^n} \binom{2n}{n}
\nonumber\\
=&\, 2^{p+1}\sum_{\sigma_j\in\{\pm 1\},\atop j=1,2,\ldots|\widetilde\bfm|_p+k-1} \Li_{(k+1,\bfm_-)^\vee} (-1,\sigma_1,\sigma_1\sigma_2,\ldots,\sigma_{|\widetilde\bfm|_p+k-2}\sigma_{|\widetilde\bfm|_p+k-1})\in \CMZV^2_{|{\bfm}|+k},
\end{align*}
where $|\widetilde\bfm|_j:=m_1+m_2+\cdots+m_j-j$.
\end{thm}
\begin{proof}
Multiplying \eqref{FII1} by $\frac{\binom{2n}{n}}{n^{k}4^n}$, summing up, and using the well-known formula (see \cite{ChenH16,Leh1985})
\begin{equation}\label{CB1}
\sum_{n=1}^\infty\frac{\binom{2n}{n}}{n 4^n }x^n
    =2\log\left(\frac{2}{1+\sqrt{1-x}}\right)=\int_{0}^x \frac{dt}{1-t+\sqrt{1-t}}\quad \forall x\in[-1,1),
\end{equation}
we have
\begin{align*}
&\quad (-1)^p \sum_{n=1}^\infty \frac{\zeta_n^\star(\bfm;x)}{n^{k}4^n} \binom{2n}{n}-\sum_{j=1}^{p} (-1)^{j} \Li_{\ola\bfm_{j}}(x)\sum_{n=1}^\infty \frac{\zeta_n^\star(\ora\bfm_{j-1})}{n^{k}4^n} \binom{2n}{n}\\
&\overset{\phantom{t\to 1-t}}=\int_0^x \frac{dt}{1-t+\sqrt{1-t}}\tx_0^{k-1}\tx_1\tx_0^{m_1-1}\cdots\tx_1\tx_0^{m_p-1} \\
&\overset{t\to 1-t}=\int_{1-x}^1 \tx_1^{m_p-1}\tx_0\cdots \tx_1^{m_1-1}\tx_0 \tx_1^{k-1}\frac{dt}{t+\sqrt{t}}\\
&\overset{\ t\to t^2\ }=2^{p+1}\int_{\sqrt{1-x}}^1 \left(\frac{2tdt}{1-t^2}\right)^{m_p-1}\frac{dt}{t}\cdots \left(\frac{2tdt}{1-t^2}\right)^{m_1-1}\frac{dt}{t} \left(\frac{2tdt}{1-t^2}\right)^{k-1}\frac{dt}{1+t}.
\end{align*}
Letting $x=1$ and noting the fact that
\begin{align*}
 \frac{2tdt}{1-t^2}=\sum_{\sigma\in\{\pm 1\}} \frac{\sigma dt}{1-\sigma t},
\end{align*}
one can obtain the desired evaluation by applying \eqref{equ:glInteratedInt}.
\end{proof}

\begin{thm}\label{EATMtV1} For any $k,p\in\N$, $\bfm=(m_1,\ldots,m_p)\in\N^p$ with $m_p\geq 2$, we have
\begin{align}\label{equ-EATMtV1}
\sum_{j=1}^{p+1} (-1)^{j-1} t(\ola\bfm_{\hskip-2pt j})\sum_{n=1}^\infty \frac{ t^\star_n(\ora\bfm_{\hskip-2pt j-1})}{n^{k}4^n}\binom{2n}{n}\in \CMZV^4_{|{\bfm}|+k}.
\end{align}
\end{thm}
\begin{proof}
Multiplying \eqref{FIIt1} by $\frac{\binom{2n}{n}}{n^{k}4^n}$, summing up, and noting the fact that
\begin{equation}\label{CB1}
\sum_{n=1}^\infty\frac{\binom{2n}{n}}{n^k 4^n }x^{2n}
    =2^k\int_{0}^x \frac{tdt}{1-t^2+\sqrt{1-t^2}} \left(\frac{dt}{t}\right)^{k-1},
\end{equation}
we have
\begin{align}\label{eq-for-mthss-mtv}
&\quad (-1)^p \sum_{n=1}^\infty \frac{t_n^\star(\bfm;x)}{n^{k}4^n} \binom{2n}{n}-\sum_{j=1}^{p} (-1)^{j} \ti_{\ola\bfm_{j}}(x)\sum_{n=1}^\infty \frac{t_n^\star(\ora\bfm_{j-1})}{n^{k}4^n} \binom{2n}{n}\nonumber\\
&\overset{\phantom{t\to 1-t}}=2^k\int_0^x \frac{tdt}{1-t^2+\sqrt{1-t^2}}\tx_0^{k-1}\om_0\tx_0^{m_1-1}\om_1\tx_0^{m_2-1}\cdots\om_1\tx_0^{m_p-1}.
\end{align}
Applying $t\to \frac{1-t^2}{1+t^2}$, we get
\begin{alignat}{4} \label{equ:changeVar1}
&\tx_0=\frac{dt}{t}\,  \to -\left(\frac{2tdt}{1+t^2}+\frac{2tdt}{1-t^2} \right)=\ty,
\qquad &
\om_0=&\, \frac{dt}{1-t^2}\to -\frac{dt}{t}=-\tx_0, \\
&\frac{tdt}{1-t^2+\sqrt{1-t^2}} \,  \to  -(\tx_i+\tx_{-i}-2\tx_{-1}),
&\qquad \om_1=&\, \frac{tdt}{1-t^2}\to -\left(\frac{dt}{t}-\frac{2tdt}{1+t^2}\right)=\tz.   \label{equ:changeVar2}
\end{alignat}
Hence, applying \eqref{equ:changeVar1} and \eqref{equ:changeVar2} to \eqref{eq-for-mthss-mtv} with $x=1$ yields
\begin{align*}
&\sum_{j=1}^{p} (-1)^{j-1} t(\ola\bfm_{j})\sum_{n=1}^\infty \frac{t_n^\star(\ora\bfm_{j-1})}{n^{k}4^n} \binom{2n}{n}\nonumber\\
&\overset{\phantom{t\to 1-t}}=2^k(-1)^{|\bfm|+k}\int_0^1 \ty^{m_p-1}\tz \cdots \ty^{m_2-1}\tz\ty^{m_1-1}\tx_0\ty^{k-1}(\tx_i+\tx_{-i}-2\tx_{-1}).
\end{align*}
Finally, the iterated integral expression \eqref{czeta-itIntExpression} of CMZVs implies \eqref{equ-EATMtV1} immediately.
\end{proof}

\begin{thm}\label{EATMZV2}
For any $k,p\in\N_0$, $\bfm=(m_1,\ldots,m_p)\in\N^p$ and $x\in [0,1]$, we have
\begin{align}\label{eq:MT}
&\sum_{n=1}^\infty \frac{\zeta_n^\star(\bfm^\vee;x)}{ n^{k+2}4^n}\binom{2n}{n} =\sum_{j=1}^p (-1)^{p-j} \Li_{\ola\bfm_{j+1}}(1-x)\sum_{n=1}^\infty \frac{\zeta_n^\star(\ora\bfm_j^\vee)}{ n^{k+2}4^n}\binom{2n}{n}\nonumber\\
&\qquad+(-1)^{p+k} 2\log(2) \Li_{(\ola\bfm)_+,\yi_k}(1-x)
 +\sum_{j=1}^k (-1)^{p+k-j}\Li_{\ola\bfm_+,\yi_{k-j}}(1-x)\su \frac{\binom{2n}{n}}{n^{j+1}4^n}\nonumber\\
&\qquad+(-1)^{p+k}2^{|\bfm|+2-p} \sum_{\sigma_j\in\{\pm 1\},\atop j=1,2,\ldots,p+k}\Li_{(\ola\bfm)_+,\yi_{k+1}}(\sigma_1\sqrt{1-x},\sigma_1\sigma_2,\ldots,\sigma_{p+k-1}\sigma_{p+k},-\sigma_{p+k}).
\end{align}
\end{thm}
\begin{proof}
One can derive the result by applying \eqref{IMP-2} and using a similar
argument as in the proof of Theorem \ref{EATMZV1}. We leave the detail to the interested reader.
\end{proof}

Letting $k=0$ in \eqref{eq:MT} yields the following corollary.
\begin{cor}\label{cor-EAS-2}
For any $p\in\N_0$, $\bfm=(m_1,\ldots,m_p)\in\N^p$ and $x\in [0,1]$, we have
\begin{align}\label{CBP3}
&\sum_{n=1}^\infty \frac{\zeta_n^\star(\bfm^\vee;x)}{n^2 4^n}\binom{2n}{n}
=(-1)^p 2\log(2)\Li_{m_p,\ldots,m_2,m_1+1}(1-x)\nonumber\\
&\quad+\sum_{j=1}^p (-1)^{p-j} \Li_{\ola\bfm_{j+1}}(1-x)
    \sum_{n=1}^\infty \frac{\zeta_n^\star(\ora\bfm_j^\vee)}{n^24^n }\binom{2n}{n}\nonumber\\
&\quad+(-1)^p2^{|{\bf m}|+2-p}\sum_{\sigma_j\in\{\pm 1\}\atop j=1,2,\ldots,p} \Li_{(\ola\bfm)_+,1}(\sigma_1\sqrt{1-x},\sigma_1\sigma_2,\ldots,\sigma_{p-1}\sigma_{p},-\sigma_{p}).
\end{align}
\end{cor}

\begin{exa}
Setting $p=1$ and $m_1=m$ gives
\begin{align*}
\sum_{n=1}^\infty \frac{\zeta_n^\star(m^\vee;x)}{n^2 4^n }\binom{2n}{n}&=\sum_{n=1}^\infty \frac{\zeta_n^\star(m^\vee)}{n^2 4^n }\binom{2n}{n}-2\log(2)\Li_{m+1}(1-x)\\
&\quad-2^{m+1} \Li_{m+1,1}(\sqrt{1-x},-1)-2^{m+1} \Li_{m+1,1}(-\sqrt{1-x},1).
\end{align*}
In particular,  taking $x=0$ we get
\begin{align*}
\sum_{n=1}^\infty \frac{\zeta_n^\star(m^\vee)}{n^2 4^n }\binom{2n}{n}=2\log(2)\ze(m+1)+2^{m+1} \ze(m+1,\bar 1)+2^{m+1} \ze(\overline{m+1},1).
\end{align*}
\end{exa}

\begin{re}
Having dealt with the case $k\geq 0$ in \eqref{eq:MT} we now provide a formula for $k=-1$.
Replacing $n$ by $k$ in \eqref{IMP-2}, summing both sides of it over $k=1,\dots,n$,
then changing $m_1\to m_1-1$, we obtain
\begin{multline}\label{d1}
\zeta_n^\star(\bfm^\vee;x)-\sum\limits_{j=1}^p (-1)^{p-j}\zeta^\star_n(\ora\bfm_j^\vee)\Li_{\ola\bfm_{j+1}}(1-x) \\
=(-1)^p\int_{0}^{1-x} \frac{1-(1-t)^n}{t}dt\, \tx_0^{m_1-1}\tx_1 \tx_0^{m_2-1}\cdots\tx_1\tx_0^{m_p-1},
\end{multline}
where $m_j\geq 1$ with $m_1\geq 2$. Here we used the fact that $(1,\bfm^\vee)=(m_1+1,m_2,\dots,m_p)^\vee$.
However, \eqref{d1} still holds for $m_1=1$, the proof of which is left to the
interested reader.
\end{re}

\begin{cor}\label{cord1}
For any $p\in\N_0$, $\bfm=(m_1,\ldots,m_p)\in\N^p$ and $x\in [0,1]$, we have
\begin{align}\label{d2}
&\sum_{n=1}^\infty \frac{\zeta_n^\star(\bfm^\vee;x)}{n4^n}\binom{2n}{n}
 =\sum_{j=1}^p (-1)^{p-j} \Li_{\ola\bfm_{j+1}}(1-x)\sum_{n=1}^\infty \frac{\zeta_n^\star(\ora\bfm_j^\vee)}{n 4^n }\binom{2n}{n}\nonumber\\
&\qquad-(-1)^p2^{|\bfm|+2-p}\sum_{\sigma_j\in\{\pm 1\}\atop j=1,2,\ldots,p-1}
\Li_{\ola\bfm_+}(\sigma_1\sqrt{1-x},\sigma_1\sigma_2,\ldots,\sigma_{p-2}\sigma_{p-1},-\sigma_{p-1}).
\end{align}
\end{cor}
\begin{proof}
Put $\tx_{-1}=dt/(-1-t)$. Multiplying \eqref{d1} by $\frac{\binom{2n}{n}}{4^n n}$, we have
\begin{align*}
&\sum_{n=1}^\infty \frac{\zeta_n^\star(\bfm^\vee;x)}{n 4^n}\binom{2n}{n}
-\sum_{j=1}^p (-1)^{p-j} \Li_{\ola\bfm_{j+1}}(1-x)\sum_{n=1}^\infty \frac{\zeta_n^\star(\ora\bfm_j^\vee)}{n 4^n }\binom{2n}{n}\\
&=(-1)^p2\int_{0}^{1-x} \frac{\log(1+\sqrt{t})}{t}dt\, \tx_0^{m_1-1}\tx_1 \tx_0^{m_2-1}\cdots\tx_1\tx_0^{m_p-1}\\
&\overset{t\to t^2}=(-1)^p2^{|\bfm|+1} \int_{0}^{\sqrt{1-x}} \frac{\log(1+t)}{t}dt\, \tx_0^{m_1-1}\frac{tdt}{1-t^2} \tx_0^{m_2-1}\cdots\frac{tdt}{1-t^2}\tx_0^{m_p-1}\\
&=(-1)^p2^{|\bfm|+2-p}\int_{0}^{\sqrt{1-x}} \frac{\log(1+t)}{t}dt\, \tx_0^{m_1-1} \left(\tx_1+\tx_{-1}\right) \tx_0^{m_2-1}\cdots\left(\tx_1+\tx_{-1}\right)\tx_0^{m_p-1}\\
&=(-1)^p2^{|\bfm|+2-p}\sum_{\sigma_j\in\{\pm 1\}\atop j=1,2,\ldots,p-1} \int_{0}^{\sqrt{1-x}} \frac{\log(1+t)}{t}dt\, \tx_0^{m_1-1} \frac{\sigma_{p-1}dt}{1-\sigma_{p-1}t} \tx_0^{m_2-1}\cdots\frac{\sigma_1dt}{1-\sigma_1t}\tx_0^{m_p-1}\\
&=(-1)^{p-1}2^{|\bfm|+2-p}\sum_{\sigma_j\in\{\pm 1\}\atop j=1,2,\ldots,p-1} \int_{0}^{\sqrt{1-x}} \tx_{-1}\tx_0^{m_1}\frac{\sigma_{p-1}dt}{1-\sigma_{p-1}t} \tx_0^{m_2-1}\cdots\frac{\sigma_1dt}{1-\sigma_1t}\tx_0^{m_p-1}.
\end{align*}
by an elementary calculation, we obtain the desired formula.
\end{proof}

In particular, letting $p=1$ and $m_1=m$ in Corollary \ref{cord1} yields
\begin{align*}
\sum_{n=1}^\infty \frac{\zeta_n^\star(m^\vee;x)}{n4^n}\binom{2n}{n}=\sum_{n=1}^\infty \frac{\zeta_n^\star(m^\vee)}{n4^n}\binom{2n}{n}+2^{m+1}\Li_{m+1}(-\sqrt{1-x}).
\end{align*}
Noticing that $(m)^\vee=(\yi_m)$ by \eqref{eq:HdualDef} and setting $x=0$ in the above equation we see that
\begin{align*}
\sum_{n=1}^\infty \frac{\zeta_n^\star(\yi_m)}{n4^n}\binom{2n}{n}=-2^{m+1}\zeta(\overline{m+1}).
\end{align*}
Therefore we get
\begin{align*}
\sum_{n=1}^\infty \frac{\zeta_n^\star(\yi_m;x)}{n4^n}\binom{2n}{n}=-2^{m+1}\zeta(\overline{m+1})+2^{m+1}\Li_{m+1}(-\sqrt{1-x}).
\end{align*}

\section{Multiple integrals associated with 5-posets}
Yamamoto first used a graphical representation to study the MZVs and MZSVs in \cite{Yamamoto2014}. In this section, we introduce the multiple integrals associated with 5-labeled posets, and express some parametric Euler-Ap\'ery type MZSVs in terms of multiple integrals associated with 5-labeled posets.

\begin{defn}
A \emph{$5$-poset} is a pair $(X,\delta_X)$, where $X=(X,\leq)$ is
a finite partially ordered set and the \emph{label map} $\delta_X:X\to \{-2,-1,0,1,2\}$.
We often omit  $\delta_X$ and simply say ``a 5-poset $X$''.

Similar to 2-poset, a 5-poset $(X,\delta_X)$ is called \emph{admissible}
if $\delta_X(x) \ne 0$ for all maximal
elements and $\delta_X(x) \ne 1,\pm 2$ for all minimal elements $x \in X$.
\end{defn}

\begin{defn}
For an admissible $5$-poset $X$, we define the associated integral
\begin{equation}
I_z(X)=\int_{\Delta_X}\prod_{x\in X}\tx_{\delta_X(x)}(t_x),
\end{equation}
where
\[\Delta_X=\bigl\{(t_x)_x\in [0,z]^X \bigm| t_x<t_y \text{ if } x>y\bigr\}\quad (z\in[0,1])\]
and
\begin{equation*}
\tx_{-2}(t)=\frac{2tdt}{1-t^2},\quad\tx_{-1}(t)=\frac{-dt}{1+t},\quad \tx_0(t)=\frac{dt}{t}, \quad \tx_1(t)=\frac{dt}{1-t},\quad \tx_{2}(t)=\frac{2dt}{1-t^2}.
\end{equation*}
Clearly,  $\tx_{-2}=\tx_1+\tx_{-1}$ and $\tx_2=\tx_1-\tx_{-1}.$
Denote by $\emptyset$ the empty 5-poset and put $I_z(\emptyset):=1$.
\end{defn}

\begin{pro}\label{prop:shuffl3poset}
For non-comparable elements $a$ and $b$ of a $3$-poset $X$, $X^b_a$ denotes the $5$-poset that is obtained from $X$ by adjoining the relation $a<b$. If $X$ is an admissible $5$-poset, then the $5$-poset $X^b_a$ and $X^a_b$ are admissible and
\begin{equation}
I_z(X)=I_z(X^b_a)+I_z(X^a_b).
\end{equation}
\end{pro}

Note that the admissibility of a $5$-poset corresponds to
the convergence of the associated integral. We use the Hasse diagrams to indicate $5$-posets,
with vertices $\circ$ and ``$\bullet\ \sigma$" corresponding to $\delta(x)=0$ and
$\delta(x)=\sigma\ (\sigma\in\{\pm 1\})$, respectively. For convenience, we
replace ``$\bullet\ 1$" by $\bullet$  and
replace ``$\bullet\ -1$" (resp. ``$\bullet\ -2$'') by ``$\bullet\ {\bar1}$'' (resp. $\bullet\ {\bar2}$'').
For example, the diagram
\[\begin{xy}
{(0,-4) \ar @{{*}-o} (4,0)},
{(4,0) \ar @{-{*}} (8,-4)},
{(8,-4) \ar @{-o}_{\bar 1} (12,0)},
{(12,0) \ar @{-o} (16,4)},
{(16,4) \ar @{-{*}} (24,-4)},
{(24,-4) \ar @{-o}_{\bar 2} (28,0)},
{(28,0) \ar @{-o} (32,4)}
\end{xy} \]
represents the $5$-poset $X=\{x_1,x_2,x_3,x_4,x_5,x_6,x_7,x_8\}$ with order
$x_1<x_2>x_3<x_4<x_5>x_6<x_7<x_8$ and label
$(\delta_X(x_1),\dotsc,\delta_X(x_8))=(1,0,-1,0,0,-2,0,0)$. For composition $\bfk=(k_1,\dotsc,k_r)$ and $\bfsi\in\{\pm 1,\pm 2\}^r$ (admissible or not),
we write
\begin{equation*}
\begin{xy}
{(0,-3) \ar @{{*}.o} (0,3)},
{(1,-3) \ar @/_1mm/ @{-} _{(\bfk,\bfsi)} (1,3)}
\end{xy}
\end{equation*}
for the `totally ordered' diagram:
\[\begin{xy}
{(0,-24) \ar @{{*}-o}_{\sigma_r} (4,-20)},
{(4,-20) \ar @{.o} (10,-14)},
{(10,-14) \ar @{-} (14,-10)},
{(14,-10) \ar @{.} (20,-4)},
{(20,-4) \ar @{-{*}} (24,0)},
{(24,0) \ar @{-o}_{\sigma_{1}}(28,4)},
{(28,4) \ar @{.o} (34,10)},
{(0,-23) \ar @/^2mm/ @{-}^{k_r} (9,-14)},
{(24,1) \ar @/^2mm/ @{-}^{k_{1}} (33,10)},
\end{xy} \]
If $k_i=1$, we understand the notation $\begin{xy}
{(0,-5) \ar @{{*}-o}_{\sigma_i} (4,-1)},
{(4,-1) \ar @{.o} (10,5)},
{(0,-4) \ar @/^2mm/ @{-}^{k_i} (9,5)}
\end{xy}$ as a single $\bullet\ {\sigma_i}$.
We see from \eqref{equ:glInteratedInt}
\begin{align}\label{Eq-MPL}
I_z\left(\ \begin{xy}
{(0,-3) \ar @{{*}.o} (0,3)},
{(1,-3) \ar @/_1mm/ @{-} _{({\bfk},\bfsi)} (1,3)}
\end{xy}\right)=\Li_{k_1,\dotsc,k_r}(\sigma_1 z,\sigma_1\sigma_2,\dotsc,\sigma_{r-1}\sigma_r),
\end{align}
where $(\sigma_1,\ldots,\sigma_r)\in\{\pm 1\}^r$.

It is clear that all multiple associated integral $I_z(\cdot)$ can be expressed in terms of multiple polylogarithm function.

\begin{thm}\label{thm-GMZSBVs} For nonnegative $k$ and positive integers $m_1,m_2,\ldots,m_p$ and real $x\in [0,1]$,
\begin{align}\label{xxxeq:MT}
&\sum_{n=1}^\infty \frac{\zeta_n^\star(\bfm^\vee;x)}{n^{k+2}4^n }\binom{2n}{n}
-\sum_{j=1}^p (-1)^{p-j} \Li_{\ola\bfm_{j+1}}(1-x)\sum_{n=1}^\infty \frac{\zeta_n^\star(\ora\bfm_j^\vee)}{n^{k+2}4^n }\binom{2n}{n}\nonumber\\
&\quad +(-1)^p \sum_{i+j=k,\atop i\geq1,j\geq 0} (-1)^i \left\{\sum_{n=1}^\infty \frac{\binom{2n}{n}}{n^{j+2}4^n} \right\}\Li_{\ola\bfm_+,\{1\}_{i-1}}(1-x)\nonumber\\
&=c_1\sum_{i+j=k,\atop i,j\geq 0}(-1)^iI_{\sqrt{1-x}}
\left(\text{\raisebox{36pt}{$
\begin{xy}
{(-8,-23) \ar @{{*}-o}^{\bar 2} (-8,-23)},
{(-8,-23) \ar @{-o} (-8,-18)},
{(-8,-18) \ar @{{*}-o}^{\bar 2} (-8,-18)},
{(-8,-18) \ar @{{*}.o} (-8,-13)},
{(-8,-13) \ar @{{*}-o}^{\bar 2} (-8,-13)},
{(-8,-13) \ar @{-o} (-8,-8)},
{(-8,-8) \ar @{{*}-o}^{\bar 2} (-8,-8)},
{(-8,-8) \ar @{-o} (-4,-13)},
{(-4,-13) \ar @{{*}-o}^{\bar 1} (-4,-13)},
{(-8,-8) \ar @{-o} (0,-4)},
{(8,-18) \ar @{{*}-o} (8,-13)},
{(8,-18) \ar @{{*}-o}_{\bar 2} (8,-18)},
{(8,-13) \ar @{{*}-o}_{\bar 2} (8,-13)},
{(8,-13) \ar @{.o} (8,-8)},
{(8,-8) \ar @{{*}-o}_{\bar 2} (8,-8)},
{(8,-8) \ar @{{*}-o} (0,-4)},
{(0,-4) \ar @{-o} (0,1)},
{(-5,-23) \ar @{} (5,-23)},
{(-7,-23) \ar @/_1mm/ @{-}_{j} (-7,-13)},
{(7,-18) \ar @/^1mm/ @{-}^{i} (7,-8)},
{(1,-4) \ar @/_1mm/ @{-}_{(\ola\bfm,{\bar 2}_{p}) } (1,1)}
\end{xy}$}} \right)
+c_2\sum_{i+j=k,\atop i,j\geq 0}(-1)^iI_{1-x}
\left(\ \text{\raisebox{33pt}{$
\begin{xy}
{(-8,-23) \ar @{{*}-o} (-8,-23)},
{(-8,-23) \ar @{-o} (-8,-18)},
{(-8,-18) \ar @{{*}-o} (-8,-18)},
{(-8,-18) \ar @{{*}.o} (-8,-13)},
{(-8,-13) \ar @{{*}-o} (-8,-13)},
{(-8,-13) \ar @{-o} (-8,-8)},
{(-8,-8) \ar @{{*}-o} (-8,-8)},
{(-8,-8) \ar @{-o} (0,-4)},
{(8,-18) \ar @{{*}-o} (8,-13)},
{(8,-18) \ar @{{*}-o} (8,-18)},
{(8,-13) \ar @{{*}-o} (8,-13)},
{(8,-13) \ar @{.o} (8,-8)},
{(8,-8) \ar @{{*}-o} (8,-8)},
{(8,-8) \ar @{{*}-o} (0,-4)},
{(0,-4) \ar @{-o} (0,1)},
{(-5,-23) \ar @{} (5,-23)},
{(-7,-23) \ar @/_1mm/ @{-}_{j} (-7,-13)},
{(7,-18) \ar @/^1mm/ @{-}^{i} (7,-8)},
{(1,-4) \ar @/_1mm/ @{-} _{(\ola\bfm,{1}_{p}) } (1,1)}
\end{xy}$}} \right),
\end{align}
where $c_1:=(-1)^p2^{|{\bf m}|+2-p}$, $c_2:=(-1)^p2\log(2)$, and\ \
$\begin{xy}
{(0,-3) \ar @{{o}.o} (0,3)},
{(1,-3) \ar @/_1mm/ @{-} _{(\bfm,\bfsi)} (1,3)}
\end{xy}$
represents the Hasse diagram obtained from \ \
$\begin{xy}
{(0,-3) \ar @{{*}.o} (0,3)},
{(1,-3) \ar @/_1mm/ @{-} _{(\bfm,\bfsi)} (1,3)}
\end{xy}$
by replacing its unique minimum $\bullet$ by $\circ$.
\end{thm}
\begin{proof} From (\ref{CB1}), by an elementary calculation, we deduce that
\begin{align}\label{CBP1}
\sum_{n=1}^\infty \frac{\binom{2n}{n}}{n^{k+2}4^n}z^n&=2\int_0^z \frac{\log\left(\frac{2}{1+\sqrt{1-t}} \right)}{t}dt\, \tx_0{k}\nonumber\\
&=\frac{2}{k!} \int_0^z \frac{\log\left(\frac{2}{1+\sqrt{1-t}} \right)\log^k\left(\frac{z}{t}\right)}{t}dt\nonumber\\
&=2\sum_{i+j=k,\atop i,j\geq 0} \frac{\log^i(z)}{i!j!} \int_0^z \frac{\log\left(\frac{2}{1+\sqrt{1-t}} \right)\log^j\left(\frac{1}{t}\right)}{t}dt.
\end{align}
Further, applying $t\to 1-t$ and using the special case $z=1$ in \eqref{CBP1} we calculate that
\begin{align}\label{CBP2}
\sum_{n=1}^\infty \frac{\binom{2n}{n}}{n^{k+2}4^n}\Big(1-(1-z)^n\Big)
&=2\sum_{i+j=k,\atop i,j\geq 0} (-1)^j\frac{\log^i(1-z)}{i!j!} \int_0^z\frac{\log\left(\frac{2}{1+\sqrt{t}}\right)\log^j(1-t)}{1-t}dt\nonumber\\&\quad-\sum_{i+j=k,\atop i\geq 1,j\geq 0} \left\{\sum_{n=1}^\infty \frac{\binom{2n}{n}}{n^{j+2}4^n} \right\}\frac{\log^i(1-z)}{i!}.
\end{align}
Multiplying (\ref{d1}) by $\frac{\binom{2n}{n}}{4^n n^{k+2}}$ and applying (\ref{CBP2}), by a direct calculation, we have
\begin{align*}
&\sum_{n=1}^\infty \frac{\zeta_n^\star(\bfm^\vee;x)}{n^{k+2}4^n }\binom{2n}{n}
-\sum_{j=1}^p (-1)^{p-j} \Li_{\ola\bfm_{j+1}}(1-x)\sum_{n=1}^\infty \frac{\zeta_n^\star(\ora\bfm_j^\vee)}{n^{k+2}4^n }\binom{2n}{n}\nonumber\\
&=-(-1)^p 2^{|\bfm|+2} \sum_{i+j=k,\atop i,j\geq 0} \frac{(-1)^j}{i!j!} \int_0^{\sqrt{1-x}} \frac{\log(1+t)\log^j(1-t^2)}{1-t^2}tdt \\
&\quad\quad\quad\quad\quad\quad\quad\quad\quad\quad\quad\quad\quad\quad\quad\quad \sha
\left(\frac{\log^i(1-t^2)}{t}dt \, \tx_0^{m_1-1}
\frac{tdt}{1-t^2}\tx_0^{m_2-1}\cdots \frac{tdt}{1-t^2}\tx_0^{m_p-1}\right)\\
&\quad+(-1)^p 2\log(2)\sum_{i+j=k,\atop i,j\geq 0} \frac{(-1)^j}{i!j!} \int_0^{{1-x}} \frac{\log^j(1-t)}{1-t}dt
\sha
\left(\frac{\log^i(1-t)}{t}dt \, \tx_0^{m_1-1} \tx_1\tx_0^{m_2-1}\cdots \tx_1\tx_0^{m_p-1}\right)\\
&\quad-(-1)^p \sum_{i+j=k,\atop i\geq1,j\geq 0} (-1)^i \left\{\sum_{n=1}^\infty \frac{\binom{2n}{n}}{n^{j+2}4^n} \right\}\frac1{i!} \int_0^{1-x} \frac{\log^i(1-t)}{t}dt\tx_0^{m_1-1} \, \tx_1\tx_0^{m_2-1}\cdots \tx_1\tx_0^{m_p-1}.
\end{align*}
Finally, according to the definition of multiple associated integral $I_z(\cdot)$
we obtain the desired evaluation using \eqref{Eq-MPL}.
\end{proof}

Letting $k=0$ in \eqref{xxxeq:MT} yields the Corollary \ref{cor-EAS-2}.

{\bf Acknowledgments.}  The first author is supported by the National Natural Science Foundation of China (Grant No. 12101008), the Natural Science Foundation of Anhui Province (Grant No. 2108085QA01) and the University Natural Science Research Project of Anhui Province (Grant No. KJ2020A0057). Jianqiang Zhao is supported by Jacobs Prize from The Bishop's School.

\end{document}